\documentclass{elsarticle}
\usepackage[margin=1in]{geometry}
\usepackage{amsmath, amssymb, latexsym, graphicx, courier, makeidx, subeqnarray, epsfig, epstopdf, tikz, textcomp, setspace, float, hyperref, booktabs, dcolumn, array, etoolbox, ifthen, bm, titlesec, caption, subcaption, environ, longtable, enumitem}
\usepackage[english]{babel}
\usepackage[titletoc,toc,title]{appendix}
\hypersetup{
	colorlinks,
	allcolors={black}
}
\makeatletter
\def\ps@pprintTitle{%
	\let\@oddhead\@empty
	\let\@evenhead\@empty
	\def\@oddfoot{}%
	\let\@evenfoot\@oddfoot}
\makeatother

\renewcommand{\thesubtable}{\Roman{tableno}}
\makeatletter
\newcommand{\subtablelabel}[1]{\stepcounter{tableno}\caption{\textbf{#1}}\setcounter{rownumbers}{0}\setcounter{rowsubnumbers}{0}\def\@currentlabel{\thesubtable}}
\newcommand{\tablelabel}{\def\@currentlabel{Table 1}\label{table}}

\newcounter{rownumbers}
\newcounter{rowsubnumbers}
\newcounter{tableno}
\newcommand\rownumber{\stepcounter{rownumbers}\arabic{rownumbers}\setcounter{rowsubnumbers}{0}\def\@currentlabel{\thesubtable.\arabic{rownumbers}}}
\newcommand\rowsubnumber{\ifthenelse{\value{rowsubnumbers}=0}{\stepcounter{rownumbers}}{}\stepcounter{rowsubnumbers}\arabic{rownumbers}(\roman{rowsubnumbers})\def\@currentlabel{\thesubtable.\arabic{rownumbers}(\roman{rowsubnumbers})}}

\newcommand\Sin{\mathrm{Sin}^{-1}}
\newcommand\Cos{\mathrm{Cos}^{-1}}
\newcommand\Tan{\mathrm{Tan}^{-1}}
\newcommand\erf{\mathrm{erf}}
\newcommand\erfc{\mathrm{erfc}}
\makeatother

\newcommand\nl{\\&&&\\\hline&&&\\\rownumber&}
\newcommand{\nlwl}[1]{\\&&&\\\hline&&&\\
	\rownumber\label{#1} & }
\newcommand\hnl{\\&&&\\\cline{3-4}&&&\\&&}
\newcommand\enl{\\&&&\\\hline}
\newcommand\stl{\hline&&&\\}
\newcommand\firstrow{
	\stl
	& \mb{y = f(x) = \L\{g(m)\}(x)} & \mb{d = g(m) = \L\{f(x)\}(m)} & \textbf{Notes}\nl}
\newcommand\firstrowwl[1]{\stl & 
	\mb{y = f(x) = \L\{g(m)\}(x)} & \mb{d = g(m) = \L\{f(x)\}(m)} & \textbf{Notes}\nlwl{#1}}
\newcolumntype{L}[1]{>{\centering$\displaystyle}p{#1}<{$}}
\newcolumntype{P}[1]{>{}p{#1}<{}}

\renewcommand{\L}[0]{\mathcal{L}}
\newcommand{\sign}{\mathrm{sign}}
\newcommand{\logit}{\mathrm{logit}}
\newcommand{\mb}[1]{\text{\boldmath{$#1$}}}

\newcommand\bbR[0]{\mathbb{R}}

\renewcommand\epsilon{\upvarepsilon}

\renewcommand{\mathbf}[1]{\ensuremath{\boldsymbol{#1}}}

\newtheorem{theorem}{Theorem}
\newdefinition{example}{Example}
\newdefinition{method}{Method}
\newproof{proof}{\textbf{Proof}}

\title{A table of Legendre-transformation pairs with methodologies for construction, authentication, and approximation of pairs}
\author[1]{Quinn T. Kolt}
\author[2]{Steven J. Kilner}
\author[1]{David L. Farnsworth}
\date{}

\affiliation[1]{organization={School of Mathematical Sciences, Rochester Institute of Technology}, 
				city={Rochester}, 
				state={New York}, 
				postcode={14623}, 
				country={USA}}
\affiliation[2]{organization={Department of Mathematics, Monroe Community College}, 
				city={Rochester}, 
				state={New York}, 
				postcode={14623}, 
				country={USA}}

\begin{document}
	\setlength{\abovedisplayskip}{5pt}
	\setlength{\belowdisplayskip}{5pt}
	\setlength{\abovedisplayshortskip}{5pt}
	\setlength{\belowdisplayshortskip}{5pt}
	\begin{abstract}
		An extensive table of pairs of functions linked by the Legendre transformation is presented. Many special functions and formulas that are used in the sciences are included in the pairs. Formulations are provided for finding the Legendre transformations of analytic functions and of polynomial approximations to other functions.
	\end{abstract}
	\begin{keyword}
		convex function, dual function, dual space, Legendre transformation, involutive property, transformation pairs table\vspace{6pt}\\
		\textit{AMS subject classification}: 44A15, 46B10, 44A20
	\end{keyword}
	
		\maketitle

	\section{Introduction}
		The \textit{Legendre transformation} of a curve in two dimensions is a curve in another space. The second curve is created by the coefficients of the original curve’s tangent and supporting lines. For the curve $y=f(x)$, the lines are written as
			$$y = mx - g(m)$$
		Properly interpreted, the curve $y = f(x)$ in $x,y$-space has the same information as the curve $d = g(m)$ in $m,d$-space. The two spaces are said to be \textit{dual} to each other, and the curves
		$$y = f(x)\text{ and }d = g(m)$$
		are \textit{dual curves}, which form a unique pair. The transformation between the two curves is called the \textit{Legendre transformation} and is written
		$$\L\{f(x)\}(m) = g(m).$$
		
		The Legendre transformation, which relates curves in a pair of two-dimensional spaces, was introduced by Adrien-Marie Legendre (1752–1833). It is widely employed in classical and statistical mechanics and thermodynamics \cite{Arnold1978,doi:10.1119/1.3119512}, and in econometrics \cite{Blackorby1978,RePEc:eee:jetheo:v:1:y:1969:i:4:p:374-396}. In mathematics, it arises in Young’s inequality \cite{hardy_1967,Dragoslov1970}, in the Clairaut differential equation \cite{Ince1944,Sternberg1954}, and in polar
		reciprocals of geometry \cite{Bix2006,Glaeser2016}.
		
		A potential limitation on the use of this transformation may be the lack of a ready source of pairs of functions that are related by the transformation. In order to remedy that, we present an extensive table of such pairs. Because the transformation is involutive or reflexive, the table is even more expansive than it might initially appear, that is, it can be read right-to-left and left-to-right. For curves that are not amenable to inclusion in the table, parametric representations and approximating curves are available and discussed here. These suggest that the Legendre transformation could have even more widespread applicability.
		
	\section{Finding the dual curve}
	Given a differentiable function $y = f(x)$, the dual curve can be expressed with $x$ as the parameter. The equation of the tangent line at $x = a$ is
	$$y = f'(a)x - (af'(a) - f(a)).$$
	Replacing the parameter $a$ with $x$ gives the parametric form of the dual curve in $m,d$-space in terms of the parameter $x$:
	\begin{align}
		m &= f'(x), \label{mdef}\\
		d &= xf'(x) - f(x).\label{ddef}
	\end{align}
	
	Substituting numerical values for $x$ into \eqref{mdef} and \eqref{ddef} gives corresponding points of the dual curve.
	If \eqref{mdef} can be solved for $x$, so that
	\begin{equation}
		x = (f')^{-1}(m),\label{solveforx}
	\end{equation}
	then \eqref{solveforx} can be substituted into \eqref{ddef} yielding an explicit form $d = g(m)$ for the dual curve, which is Method 1 in Section 3. 
	
	\begin{example}[an example $y = f(x)$ for which \eqref{mdef} cannot be analytically solved for $x$]
		Consider the composite function $y = f(x) = \sin(x^{2})$ for $x\in\bbR$. Then, $y' = f'(x) = m = 2x\cos(x^{2})$ for $m\in\bbR$ and $d = mx - y = 2x^{2}\cos(x^{2}) - \sin(x^{2})$. Thus, a parametric representation of the dual curve is
		\begin{align*}
			m &= 2x\cos(x^{2})\\
			d &= 2x^{2}\cos(x^{2}) - \sin(x^{2}).
		\end{align*}
	\end{example}
	
	Opting for a parametric portrayal can be very useful for sums of functions, as well.
	
	\begin{example}[sum of functions]
		Consider $y = f(x) = \sin(x^{2}) - x^{3} + e^{x}$ for $x\in\bbR$, then $y'=f'(x) = m = 2x\cos(x^{2}) - 3x^{2} + e^{x}$ for $m\in\bbR$ and $d=mx-y = 2x^{2}\cos(x^{2}) - \sin(x^{2}) - 2x^{3} + (x-1)e^{x}$.	
	\end{example}

	In certain situations, which are described in Methods 1–4 in Section 3, explicit formulations of the dual curve are available and yield the pairs of functions, such as those in \ref{table}, which is in \ref{app}.

	\section{Methodologies for finding and verifying the Legendre-transformation pairs in \ref{table}}
	The following theorem contains three facts that are helpful. Theorem \ref{thmleg}\ref{thmslope} says that $x$ plays the role of slope in the dual space. Theorem \ref{thmleg}\ref{thmmaint} says that convexity is maintained by the transformation \cite{doi:10.1119/1.3119512,Rockafellar70convexanalysis}. Theorem \ref{thmleg}\ref{thmrefl} says that we can read \ref{table} in \ref{app} left-to-right and right-to-left \cite{Rockafellar70convexanalysis,Lay2007}.
	
	\begin{theorem}\label{thmleg}
		If $f(x)$ is a differentiable function and $\L\{f(x)\}(m) = g(m)$, then
		\begin{enumerate}[label=\textup{(\alph*)}]
			\item\label{thmslope} $x$ is slope in the dual curve: the independent variables $x$ and $m$ are slopes of the corresponding dual curves. In particular where both curves possess first derivatives and nonzero second derivatives,
			$$m = f'(x)\text{ and }x = g'(m).$$
			\item \label{thmmaint} Maintenance of strict convexity and strict concavity: corresponding portions of $f(x)$ and $g(m)$ have the same concavity, in particular, where both curves possess second derivatives,
			$$g''(m) = 1/f''(x).$$
			\item\label{thmrefl} Reflective or involutive property of the transformation:
			$$\L\{\L\{f(x)\}(m)\}(a) = f(a).$$
		\end{enumerate}
	\end{theorem}
	\begin{proof}
		For \ref{thmslope}, by definition $m = y'(x) = f'(x)$. Let $\L\{f(x)\}(m) = g(m)$ be given by the parametric equations
		\begin{align*}
			m &= f'(a)\\
			d &= g(m(a)) = af'(a) - f(a).
		\end{align*}
		Thus, $\frac{dm}{da}=f''(a)$. Then,
		$$\frac{dg}{dm} = \frac{dg}{da}\frac{da}{dm} = (f'(a) + af''(a) - f'(a))\frac{1}{f''(a)} = a$$
		for any value of $a$.
		
		For \ref{thmmaint}, apply $\frac{d}{da}$ to $g'(m) = a$ to obtain $g''(m) \frac{dm}{da} = g''(m) f''(a) = 1$.
		
		For \ref{thmrefl}, in $m,d$-space, for any given value of $a$, the equation of the corresponding tangent line is
		$$d = a(m - f'(x)) + af'(a) - f(a) = am - f(a).$$
		Therefore, $\L\{g(m)\}(a) = f(a)$ as claimed.
	\end{proof}
	Theorem \ref{thmleg}\ref{thmslope} provides the locations of points where first derivatives are zero, if any such points exist, in terms of values of the dual function.  For $y = f(x)$, from $m = f'(x) = 0, x = g'(m) = g'(0)$ and $y = f(x(m)) = f(x(0)) = 0x - g(0)$. Any zeros of the derivative occur at $(x,y) = (g'(0),-g(0))$. Similarly, any zeros of the derivative of $d = g(m)$ occur at $(m,d) = (f'(0),-f(0))$.

	\ref{table}, which is in \ref{app}, in contains pairs of curves in the two dual spaces that are related by the Legendre transformation. There are four main ways for determining the entries of the table. We illustrate each method with an example. For a given curve, differing methods might be applied to various portions of the curve. Usually, more than one method can be successfully employed.
	
	From the equation of a tangent or supporting line $y = mx - g(m)$, the Legendre transformation of $y = f(x)$ is 
	\begin{equation}
		\L\{f(x)\}(m) = g(m) = mx(m) = f(x(m)).\label{lif}
	\end{equation}
	
	\begin{method}[using lines of tangency and the inverse of the derivative of $y = f(x)$]\label{methinv}
		For differentiable functions $f(x)$, whose derivatives have an inverse, from Theorem \ref{thmleg}\ref{thmslope},
		\begin{equation}
			x(m) = (f')^{-1}(m).\label{inv}
		\end{equation}
		Substituting \eqref{inv} into \eqref{lif} gives
		\begin{eqnarray}
			d = g(m) = \L\{f(x)\}(m) = mx(m) - f(x(m)) = m(f')^{-1}(m) - f((f')^{-1}(m)).\label{mff}
		\end{eqnarray}
	\end{method}
	\begin{example}[using lines of tangency in Method \ref{methinv}]
		Consider the function
		\begin{equation}
			y = f(x) = \frac{x^{3}}{3}\label{tangex}
		\end{equation}
		for $x\in\bbR^{+}$. Placing $x(m) = (f')^{-1}(m) = m^{1/2}$ into \eqref{mff} gives the dual curve
		\begin{equation}
			d = g(m) = \L\{f(x)\}(m) = m(m^{1/2}) - \frac{(m^{1/2})^{3}}{3} = \frac{2m^{3/2}}{3}\label{Lx33}
		\end{equation}
		for $m\in\bbR^{+}$. The convexity of \eqref{tangex} and \eqref{Lx33} illustrates Theorem \ref{thmleg}\ref{thmmaint}. Equations \eqref{tangex} and \eqref{Lx33} are Entry \ref{xpp} of \ref{table} in \ref{app} with $p=3$ and $q=3/2$.
	\end{example}

	\begin{method}[using lines of support]\label{methsupp}
		For lines of support, the negation of the lines' $y$-intercept for each value of $m$ is the value of $g(m)$.
	\end{method}
	\begin{example}
		Consider the convex function
		\begin{equation}
			y = f(x) = \begin{cases}
				-x + 2 & x\leq 1,\\
				2x - 1 & 1 < x.
			\end{cases}\label{supportex}
		\end{equation}
		The transformation of the line segment $y = -x + 2$ for $x < 1$ is the point $(m,d) = (-1,-2)$. The transformation of the line segment $y = 2x - 1$ for $1 < x$ is the point $(m,d) = (2,1)$. At $x = 1$, the lines of support of \eqref{supportex} are $y = mx + (1 - m)$ for $-1 < m < 2$, so that the transformation of the point $(x,y) = (1,1)$ is the open line segment $d = m - 1$ with $-1 < m < 2$, whose end points are $(-1,-2)$ and $(2,1)$. Thus,
		$$d = g(m) = m - 1\text{ for }-1\leq m\leq 2.$$
		This example illustrates how line segments and points that are vertices transform into each other.
	\end{example}
	\begin{method}[using the integral form]\label{methint}
		The integral form of the Legendre transformation $y = f(x)$, whose derivative has an inverse, is
		\begin{equation}
			g(m) = \L\{f(x)\}(m) = \int (f')^{-1}(m)dm + C,\label{gint}
		\end{equation}
		which can be confirmed by showing the equality of the derivatives of \eqref{mff} and \eqref{gint} with respect to
		$m$. From Theorem \ref{thmleg}\ref{thmslope}, $m = f'(x) = f'(g'(m))$, and, thus, $(f')^{-1}(m) = g'(m)$. The constant of
		integration can be determined from the slope and negation of the $y$-intercept of the tangent line at a point of $y = f(x)$. Selecting the point where $m = m_{0}$, \eqref{gint} can be expressed
		\begin{equation}
			g(m) = \int_{m_{0}}^{m}(f')^{-1}(t)dt + g(m_{0}).\label{defint}
		\end{equation}
	\end{method}
	\begin{example}[using the integral form in Method \ref{methint}]
		For $y = f(x) = \cos x$ with $-\pi/2 < x < \pi/2$, we have $f'(x) = -\sin x = m$ and $(f')^{-1}(m) = -\Sin m$ for $-1 < m < 1$. The integral of the inverse sine function is
		$$\int(-\Sin m)dm = -m\Sin m - \sqrt{1 - m^{2}} + C.$$
		From \eqref{defint}, employing the point $(0, 1)$ of $y = f(x)$, where $m_{0} = 0$ and the negation of the $y$-intercept of the tangent line is $g(m_{0}) = g(0) = -1$, obtain
		\begin{align*}
			g(m) = \L\{\cos x\}(m) &= \int_{0}^{m} (-\Sin t)dt + (-1)\\
			&= \left(-m\Sin m - \sqrt{1 - m^{2}}\right) - \left(-0\Sin 0 - \sqrt{1 - 0^{2}}\right) - 1\\
			&= -m\Sin m - \sqrt{1 - m^{2}}
		\end{align*}
		for $-1<m<1$, which is Entry \ref{cos} of \ref{table} in \ref{app}.
	\end{example}
	\begin{method}[using a limit of the negation of $y$-intercepts]\label{methsup}
		For a convex function $y = f(x)$, or a portion of the function that is convex, the lines of support or tangency for each value of $m$ give
		\begin{equation}
			d = g(m) = \L\{f(x)\}(m) = \sup_{x}\{mx - f(x)\},\label{sup}
		\end{equation}
		where $\sup$ denotes supremum, which is also called the least upper bound. Recall that all convex functions are continuous \cite{Lay2007,Tiel1984ConvexAA,Bauschke2017}. This method determines the line with the largest $y$-intercept among all lines with slope $m$ and below the convex function. For the strictly concave function $y = f(x)$, write
		$$d = g(m) = \L\{f(x)\}(m) = \inf_{x}\{mx - f(x)\},$$
		where $\inf$ denotes infimum, which is also called the greatest lower bound.
	\end{method}
	\begin{example}[using the interpretation as the limit in Method 4]
		To derive Entry \ref{ex} of \ref{table} in \ref{app}, where $y = f(x) = e^{x}$ is convex for $x\in\bbR$ and has positive slope, select any positive value for $m$ and set
		$$\frac{d}{dx}(mx - e^{x}) = m - e^{x} = 0.$$
		Then, $x = \ln m$. Substituting this into \eqref{sup} gives
		$$d = g(m) = \L\{f(x)\}(m) = m(\ln m) - e^{\ln m} = m\ln m - m$$
		for $m > 0$.
	\end{example}
	
	Method \ref{methsup} is effective for deriving properties in Part \ref{props} of the table. For example, for Entry \ref{fpa} of the table, for convex functions $y$,
	$$\L\{y(x) + a\}(m) = \sup_{x}\{mx - (y(x) - a)\} = \sup_{x}\{mx - y(x)\} - a = \L\{y(x)\}(m) - a.$$
	
	With most modern plotting software, such as Desmos, GeoGebra, and MATLAB, one can use the following test in order to rapidly verify that functions $f$ and $g$ are Legendre-transformation pairs on stated domains. From \eqref{ddef} and \eqref{mdef}, the functions are transformations of each other on the $x$-domain where
	$$y = xf'(x) - f(x) - g(f'(x))$$
	is the zero function. The $m$-domain for the pair is the range of
	$$y = f'(x)$$
	over the $x$-domain that was just confirmed. Many other options for similar authentications are available. These checks of the accuracy of the entries of the table do not replace this section's methods for creating entries.
	
	\section{The table of Legendre-transformation pairs}
	Because the purpose of \ref{table} in \ref{app} is to serve as a quick and easy source of
	Legendre-transformation pairs, there is some redundancy. For example, Entry \ref{combo} is an amalgam of Entries \ref{scaleout}-\ref{shiftin}. Entry \ref{ex} is $\L\{e^{x}\}(m) = m\ln m - m$ and Entry \ref{xln}, reading right-to-left, is
	$\L\{e^{x-1}\} = m\ln m$; Entry \ref{xln} can be obtained from Entry \ref{ex} using Entry \ref{shiftin} with $f(x) = e^{x}$
	and $a = -1$. Some entries are special cases of others. For example, Entry \ref{ex} is Entry \ref{ax} with $a=e$.
	
	The properties in Part \ref{props} are aids for expanding the applicability of the remainder of the
	table to new functions that are related to table entries by the properties. For example, to obtain
	the transformation of $\alpha x^{3} + \beta x^{2}$ for any nonzero $\alpha$ and $\beta$, consider Entry \ref{x3x2}, which gives the
	transformation of $x^{3}/3 + x^{2}/2$. Use Entry \ref{scalein} with $a = 3\alpha/(2\beta)$ to obtain the transformation of $(3\alpha x/(2\beta))^{3}/3 + (3\alpha x/(2\beta))^{2}/2$, then apply Entry \ref{scaleout} with $a = (2\beta)^{3}/(3\alpha)^{2}$ to find
	$$\frac{(2\beta)^{3}}{(3\alpha)^{2}}\left(\frac{1}{3}\left(\frac{3\alpha x}{2\beta}\right)^{3} + \frac{1}{2}\left(\frac{3\alpha x}{2\beta}\right)^{2}\right) = \alpha x^{3} + \beta x^{2}.$$
	
	Part \ref{props} supplies relationships between operations on variables and functions in one space
	and operations in the dual space. For example, Entry \ref{scaleout} shows that multiplying a function by $a >
	0$ in one space corresponds to \textit{epi-multiplication} by $a$ in the other space; epi-multiplication by $a$
	is dividing the independent variable by a and simultaneously multiplying the function by $a$. In
	applications, the variables and functions can have physical, economic, or other meanings and the
	properties describe how algebraic and other operations in one space influence variables and
	functions, and hence their meanings, in the other space. Entries \ref{scaleout}-\ref{combo} reflect how changes of
	coordinates in one space alter the transformed function.
	
	The infimal convolution or epi-sum in Entries \ref{sumfs} and \ref{px2} is discussed and used in \cite{Bauschke2017}
	and \cite{Tiel1984ConvexAA}. For differentiable functions, the infimum is the minimum that can be found by
	differentiation with respect to $t$, so that implementation of those two entries can be relatively
	straightforward.
	Parts \ref{alg}, \ref{exp}, and \ref{trig} contain algebraic, logarithmic and exponential, and trigonometric functions,
	respectively. The special functions in Part \ref{spe} are extensively used in many areas of physics and
	other disciplines \cite{bell2004special,lebedev1972special,50831}.
		
	The \textit{Lambert $W$ function}, which is the solution $w = w(x)$ of the transcendental equation $we^{w} = x$ and a subject of Part \ref{spe}, has become a much-used tool in physics \cite{Corless1996,doi:10.1080/07468342.2004.11922095} and other
	disciplines \cite{10.1214/11-AOAS457,10.1214/14-AOAS790,10.2112/JCOASTRES-D-17-00181.1,doi:10.4169/074683410X480276}. 
	Its history is contained in \cite{Hayes2005}, and many properties are in \cite{Corless1996} and \cite{50831}.
	Another name for the Lambert $W$ function is the \textit{product logarithm}. Numerical values are
	available in \textit{Mathematica}, \textit{Maple}, \textit{Matlab}, and \textit{Wolfram Alpha}. The function $y = W(x)$ is defined on the complex plane and has two real branches. The principal branch is the function $y = W_{0}(x) =
	W_{p}(x)$, which has domain $x\in (-1/e,\infty)$, contains the points $(-1/e,-1)$ and $(0,0)$ and is increasing,
	concave, and positive for $x\in [0,\infty)$. The other branch is the function $y = W_{-1}(x) = W_{m}(x)$, which
	has domain $x\in [-1/e,0)$ , contains the point $(-1/e,-1)$, is decreasing, convex for
	$x\in (-1/e,-2/e^{2})$, concave for $x\in (-2/e^{2},0)$, and asymptotically approaches the negative $y$-axis.
	
	\section{Using \ref{table}}
	For the first part of Entry \ref{sin} of \ref{table} in \ref{app}, the full transformation for all $m\in\bbR$ is
	\begin{equation}
		\L\{\sin(x)\}(m) = \begin{cases}
			(\pi/2)m-1 & \text{ for }m\leq 0,\\
			m\Cos m - \sqrt{1 - m^{2}} & \text{ for }0<m<1,\\
			0 & \text{ for }1\leq m,
		\end{cases}\label{usageex}
	\end{equation}
	by using the supporting lines at the endpoints of the domain of $y = f(x)$. Method \ref{methsupp} is employed to find the extension in \eqref{usageex}, which is outside $0 < m < 1$. Many entries have the aspect that only the nonlinear portion is displayed in the table, but line segments to the left and/or right can be taken to complete the function, as in \eqref{usageex}. When they have no physical or other meaning, these line segments are ignored. For simplicity, they are omitted from the table.
	
	Often, it is assumed that all functions whose Legendre transformations are sought are convex. Then, Theorem \ref{thmleg}\ref{thmmaint} implies that the Legendre transformation is convex, as well. Differentiable convex functions have derivatives that are invertible because the derivatives are monotonic and single valued \cite{Tiel1984ConvexAA}. Some entries of the table are split into parts, which depend on whether the portion of the function is convex or concave.
	
	The transformation can be applied to functions that are part convex and part concave, but there will most likely be complications and, sometimes, unexpected results. Theorem \ref{thmleg}\ref{thmmaint} shows that, if a function has a second derivative that is zero at a point, the corresponding point on the dual curve has an undefined second derivative. Example \ref{conex} illustrates some of these ideas.
	
	\begin{example}[function that has both concave and convex portions]\label{conex}
		The function $y = f(x) = \sin x$ for $-\pi/2 < x < \pi/2$ is in Entry \ref{sin}. The function is convex on the left and concave on the right. The function $y = f(x)$ and its transformation are displayed in Figures \ref{sinfig} and \ref{Lsinfig}. As $y = f(x)$ is traced left to right, starting at $(x,y) = (-\pi/2,-1)$, the corresponding points of its transformation are traced downward, starting at $(m,d) = (0,1)$. For each value of $x\neq 0$, the tangent lines to the sine function at points with coordinates $x$ and $-x$ have the same slope $m$ and different intercepts, which makes the transformation in Figure \ref{Lsinfig} a curve, not a function. The point of inflection $(0,0)$ in Figure \ref{sinfig} corresponds to the cusp at $(1,0)$ in Figure \ref{Lsinfig}. The symmetry in Figure \ref{sinfig} about the $m$-axis illustrates Entry \ref{odd}.
	\end{example}
	
	The Legendre transformation changes the function $f$ into the function $g$ by changing the independent variable from $x$ to $m$. In physical settings, there may be many other variables present, but the underlying process is unchanged. The transformation’s appearance might change because partial derivatives are employed, in order that other variables are held constant \cite{Alford2019}.
	
	Most domains are presented as open sets, but, in very many cases, boundary points can be
	filled in simply by substitution or by a limiting process.
	
	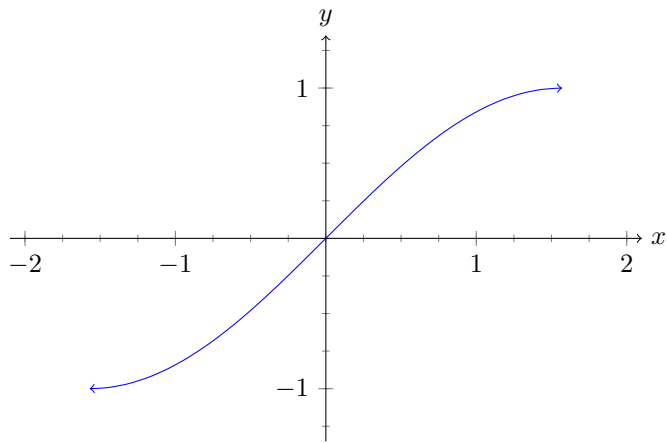
\begin{figure}[ht]
		\begin{center}
			\begin{tikzpicture}[scale=2]
				\foreach \i in {-2,-1.75,...,2} {
					\draw [very thin,gray] (\i,-.025) -- (\i,.025);
				}
				\foreach \i in {-1.25,-1,...,1.25} {
					\draw [very thin,gray] (-.025,\i) -- (.025,\i);
				}
				\draw [thin,gray] (2,-.05) -- (2,.05) node[below, yshift=-.2cm, black] {$2$};
				\draw [thin,gray] (1,-.05) -- (1,.05) node[below, yshift=-.2cm, black] {$1$};
				\draw [thin,gray] (-1,-.05) -- (-1,.05) node[below, yshift=-.2cm, black] {$-1$};
				\draw [thin,gray] (-2,-.05) -- (-2,.05) node[below, yshift=-.2cm, black] {$-2$};
				\draw [thin,gray] (-.05,1) -- (.05,1) node[left, xshift=-.2cm, black] {$1$};
				\draw [thin,gray] (-.05,-1) -- (.05,-1) node[left, xshift=-.2cm, black] {$-1$};
				
				\draw[->] (-2.1, 0) -- (2.1, 0) node[right] {$x$};
				\draw[->] (0, -1.35) -- (0, 1.35) node[above] {$y$};
				\draw[scale=1, domain=-1.57:1.57, smooth, variable=\x, blue, <->] plot ({\x}, {sin(180*\x/pi)});
			\end{tikzpicture}
		\end{center}
		\caption{$y = f(x) = \sin x$ for $-\pi/2 < x < \pi/2$}\label{sinfig}
	\end{figure}

	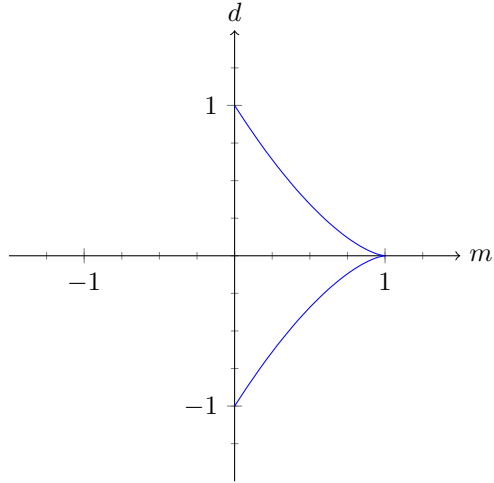
\begin{figure}[ht]
		\begin{center}
			\begin{tikzpicture}[scale=2]
				\foreach \i in {-1.25,-1,...,1.25} {
					\draw [very thin,gray] (\i,-.025) -- (\i,.025);
				}
				\foreach \i in {-1.25,-1,...,1.25} {
					\draw [very thin,gray] (-.025,\i) -- (.025,\i);
				}
				\draw [thin,gray] (1,-.05) -- (1,.05) node[below, yshift=-.2cm, black] {$1$};
				\draw [thin,gray] (-1,-.05) -- (-1,.05) node[below, yshift=-.2cm, black] {$-1$};
				\draw [thin,gray] (-.05,1) -- (.05,1) node[left, xshift=-.2cm, black] {$1$};
				\draw [thin,gray] (-.05,-1) -- (.05,-1) node[left, xshift=-.2cm, black] {$-1$};
				
				\draw[->] (-1.5, 0) -- (1.5, 0) node[right] {$m$};
				\draw[->] (0, -1.5) -- (0, 1.5) node[above] {$d$};
				\draw[scale=1, domain=0:1, smooth, variable=\m, blue] plot ({\m}, {-pi/180*\m*asin(sqrt(1 - \m*\m)) + sqrt(1 - \m*\m)});
				\draw[scale=1, domain=0:1, smooth, variable=\m, blue] plot ({\m}, {pi/180*\m*asin(sqrt(1 - \m*\m)) - sqrt(1 - \m*\m)});
			\end{tikzpicture}
		\end{center}
		\caption{$d = g(m) = \L\{\sin(x)\}(m)$ from Entry \ref{sin} is the Legendre transformation of the function in Figure \ref{sinfig}}\label{Lsinfig}
	\end{figure}
	
	\section{Approximating polynomials}
		The Lagrange inversion or Lagrange-B\"urmann theorem pertains to finding dual functions. It says
		that, given an analytic function $y = r(x)$ with Taylor series about $x = x_{0}$ and nonzero derivative at
		$x = x_{0}$, a Taylor series about $y_{0} = r(x_{0})$ of the inverse function $x = s(y)$ can be found and has a
		non-zero radius of convergence. Contour integration and various other techniques can be
		employed \cite{50831,https://doi.org/10.1002/zamm.19630430916,Abramowitz1972}. \cite{dienes1957taylor} performs inversion by the method of indeterminate
		coefficients.
		
		For the present application, $y = r(x) = f'(x)$ and $x = s(y) = g'(y)$, which are inverse
		functions by Theorem \ref{thmleg}\ref{thmslope}. Because all derivatives may not exist for functions of interest, we
		find approximating polynomials in the dual space. These polynomials would be terms of the
		Taylor series. Interesting ideas, especially about this type of potentially limited use of Lagrange
		inversion, can be found in \cite{Grossman2005}, \cite{Johnson2002}, \cite{Krattenthaler1988}, and \cite{Traub1962}.
		
		Polynomials that approximate $d = g(m)$ near $m = m_{0}$ can be found without first finding the
		function $g$. Each term of the approximating polynomial, which are the terms of the Taylor
		expansion of $g$, is a function of the derivatives of $y = f(x)$ at $x = x_{0}$ up to the same degree as the
		polynomial in $m, d$-space. The terms of the polynomial approximation of $g$ about $m=m_{0}$ are
		determined by the terms of the polynomial of the same degree of approximation of $f(x)$ at $x = x_{0}$.
		Besides differentiability of $f$ at $x = x_{0}$ to the degree desired, it is required that $f''(x_{0}) \neq 0$.
		
		From the definition of $m = f'(x)$, 
		$$m_{0} = f'(x_{0}).$$
		The defining equation of $d = g(m)$ is from the tangent line $y=mx - d$ so that $g = mx - f$ and
		$$g(m_{0}) = x_{0}f'(x_{0}) - f(x_{0}).$$
		From Theorem \ref{thmleg}\ref{thmslope}, $g'(m) = x$, so that
		$$g'(m_{0}) = x_{0}.$$
		From Theorem \ref{thmleg}\ref{thmmaint}, $g''(m) = 1/f''(x)$, so that
		$$g''(m_{0}) = \frac{1}{f''(x_{0})}.$$
		A recursion formula for $g^{(n+1)}(m)$ for integers $n\geq 2$ is available. Beginning with $g''(m) = 1/f''(x)$, differentiate with respect to $x$ to obtain
		$$\frac{dg''(m)}{dm}\frac{dm}{dx} = g'''(m)f''(x) = \frac{d}{dx}\frac{1}{f''(x)} = -\frac{f'''(x)}{(f''(x))^{2}}$$
		and 
		$$g'''(m) = \frac{1}{f''(x)}\frac{d}{dx}\frac{1}{f''(x)} = -\frac{f'''(x)}{(f''(x))^{3}},$$
		so that
		$$g'''(m_{0}) = -\frac{f'''(x_{0})}{(f''(x_{0}))^{3}}.$$
		Continuing to differentiate similarly, we obtain
		\begin{align*}
			g^{(4)}(m) &= \frac{1}{f''(x)}\frac{d}{dx}\left(-\frac{f'''(x)}{(f''(x))^{3}}\right) = \frac{3(f'''(x))^{2} - f''(x)f^{(4)}(x)}{(f''(x))^{5}},\\
			g^{(5)}(m) &= \frac{1}{f''(x)}\frac{d}{dx}\left(\frac{3(f'''(x))^{2} - f''(x)f^{(4)}(x)}{(f''(x))^{5}}\right) = \frac{10f''(x)f'''(x)f^{(4)}(x) - 15(f'''(x))^{3} - (f''(x))^{2}f^{(5)}(x)}{f''(x)^{7}},
		\end{align*}
		and so forth. The right-hand sides can be evaluated at $x=x_{0}$. To obtain the next higher derivative, and thus the next term, divide the derivative of the current derivative as a function of $x$ by $f''(x)$ according to the recursion formula
		$$(g^{(n+1)}(m))(x) = \frac{1}{f''(x)}\frac{d}{dx}(g^{(n)}(m))(x).$$
	
		\begin{example}[product of functions $y=x\sin x$]
			The Taylor series for $y = f(x) = x\sin x$ about $x=0$ is
			\begin{align*}
				y = f(x) = \sum_{i=0}^{\infty} f^{(i)}(0)\frac{x^{i}}{i!} = x\sin x = x\sum_{i=0}^{\infty}(-1)^{i} \frac{x^{2i+1}}{(2i+1)!} = \sum_{i=1}^{\infty} (-1)^{i-1}\frac{x^{2i}}{(2i-1)!},
			\end{align*}
			which converges for all real numbers \cite{50831}. Thus, $m_{0} = f'(0)$, the odd derivatives of $f$ at $x=0$ are zero, and, for all nonnegative integers $n$, $f^{(2n)}(0) = (-1)^{n+1}(2n)$. Then, at $m = m_{0} = 0$, $g(0) = 0, g'(0) = 0, g''(0) = \frac{1}{2}, g'''(0) = 0,$ and
			$$g^{(4)}(0) = \frac{3(f'''(0))^{2} - f''(0)f^{(4)}(0)}{(f''(0))^{5}} = \frac{3(0)^{2} - (2)(-4)}{(2)^{5}} = \frac{1}{4},$$
			so the first five terms of the approximating polynomial at $m=0$ to the Legendre transformation of $x\sin x$ is
			$$0 + 0m + \frac{1}{2}\frac{m^{2}}{2!} + 0\frac{m^{3}}{3!} + \frac{1}{4}\frac{m^{4}}{4!} = \frac{m^{2}}{4} + \frac{m^{4}}{96}.$$
		\end{example}

	\section{Closing comments}
		Clairaut’s differential equation
		\begin{equation}
			y = xy' + h(y')\label{clair}
		\end{equation}
		is a expression of the Legendre transformation between two differentiable functions \cite{Ince1944,Sternberg1954}. This can be seen by replacing $y$ by $f$, $y'$ by $m$, and the function $h$ by the function $-g$. The singular solution of \eqref{clair} is the Legendre transformation of the function $g(m)$ and vice versa by the reflectivity property in Theorem \ref{thmleg}\ref{thmrefl}. The general solution of \eqref{clair}, which contains a constant of integration as a parameter, is the tangent lines to the singular solution \cite{Rainville1964,zill2008differential}. \ref{table} in \ref{app} is an integral table for Clairaut’s differential equation \cite{Sternberg1954}. Clairaut’s equation and the Legendre transformation are displayed in the formula for integration by parts with transformation are displayed in the formula for integration by parts with $\int f'(x)dx = mx - \int g'(m)dm$.
		
		Besides $m,d$-space, there are many spaces that are dual to $x,y$-space. Each dual space is formulated in terms of the coefficients of a standard form for the equations of lines that serve as tangent and supporting lines to curves in $x,y$-space \cite{kilner2021tangent}. Here, $m$ and $d$ from $y = mx - g(m)$ yield curves $d = g(m)$ that are dual to $y = f(x)$, which has the lines as supporting or tangent lines. Hence, $f(x)$ and $g(m)$ are Legendre transformation pairs.
		
		Using $y = mx + b(m)$ to give points in $m,b$-space is often called the Legendre transformation. This transformation is intuitive, because the $y$-intercept $b$ is a coordinate in the dual space, but the reflexivity property is lost. In order to use the table for this alternative definition, the adjustment can be made between $b$ and $d$ \cite{Alford2019,kilner2021tangent}.
		
		Another example is $u,v$-space, where the coefficients $u$ and $v$ of
		\begin{equation}
			ux+vy = 1 \label{uv}
		\end{equation}
		in $x,y$-space give the dual curves $v = w(u)$ \cite{Thompson1996}. The pairs in the table can be converted to relationships for the $u$ and $v$ coefficients using the identities $m = –u/v$ and $d = 1/v$ \cite{kilner2021tangent}.  Different lines in $x,y$-space may have no representations, thus losing various portions of curves. In $x,y$-space, vertical lines have undefined slope $m$ and, therefore, neither $m,d$-representation nor $m,b$-representation, and lines through the origin of $x,y$-space have no $u,v$-representation because the lines cannot be written as \eqref{uv}. Accommodations or adjustments in order to include omitted lines can be made by introducing points at infinity, but that can sometimes defeat the goal of simple or physical representations.

	\pagebreak
	
	\appendix
	\section{\ Table 1: Legendre-transformation pairs}\label{app}\tablelabel
	\begin{table}[H]
		\begin{subtable}{\textwidth}
			\subtablelabel{Properties}\label{props}
			\begin{center}$
				\begin{tabular}{|c|L{5cm}|L{5cm}|P{4cm}|}
					\firstrowwl{scaleout}
					af(x) & ag\left(\frac{m}{a}\right) & $a\neq 0$\nlwl{scalein}
					f(ax) & g\left(\frac{m}{a}\right) & $a\neq 0$\nlwl{fpa}
					f(x) + a & g(m) - a & \nlwl{shiftin}
					f(x + a) & g(m) - am &\nlwl{combo}
					cf(sx+t) + bx + a & cg\left(\frac{m - b}{cs}\right) - t\frac{m - b}{s} - a & $c\neq 0, s\neq 0$\nl
					f^{-1}(x) & -mg\left(\frac{1}{m}\right) & $m\neq 0$\nl
					f'(x) & mf''^{-1}(m) - f'(f''^{-1}(m)) &\nl
					\int_{a}^{x} f(t)dt & mf^{-1}(m) - \int_{a}^{f^{-1}(m)} f(t)dt&\nlwl{sumfs}
					(f_{1}\square f_{2})(x) =\newline \inf_{t}\{f_{1}(x-t) + f_{2}(t)\} & \L\{f_{1}(x)\}(m) + \L\{f_{2}(x)\}(m) & {$f_{1}(x)$ and $f_{2}(x)$ convex, $(f_{1}\square f_{2})(x)$ is the \textit{infimal convolution} or \textit{epi-sum}}\nlwl{px2}
					f(x) + \frac{x^{2}}{2} & g(m)\square \frac{m^{2}}{2} & $f(x)$ convex\nl
					f(x)\text{ even, i.e., }f(-x) = f(x) & g(m)\text{ even} & \nlwl{odd}
					f(x)\text{ odd, i.e., }f(-x) = -f(x) & g^{-1}(m)\text{ even} & \enl
				\end{tabular}$
			\end{center}
		\end{subtable}
	\end{table}

	\begin{table}[H]
		\begin{subtable}{\textwidth}
			\subtablelabel{Algebraic Functions}\label{alg}
			\begin{center}$
				\begin{tabular}{|c|L{5cm}|L{5cm}|P{4cm}|}
					\firstrowwl{xpp}
					\frac{x^{p}}{p} & \frac{m^{q}}{q} & $0 < x, 0 < m$,\newline $\displaystyle \frac{1}{p} + \frac{1}{q} = 1, p\neq 0\text{ or }1$
					\nl\sqrt{1-x^{2}}  & -\sqrt{1 + m^{2}} & $|x| < 1$\nl
					(1-x^{p})^{1/p}  & -(1 + (-m)^{q})^{1/q} & $0\leq x < 1, m\leq 0, \displaystyle\frac{1}{p} + \frac{1}{q} = 1, 1 < p$
					\nl\sqrt{x^{2} - 1}  & \sqrt{m^{2} - 1} & $1<|x|, 1<|m|$
					\nlwl{x3x2}\frac{x^{3}}{3} + \frac{x^{2}}{2}  & -\frac{1}{12} - \frac{m}{2} - \frac{(1 + 4m)^{3/2}}{12} & $x < -\frac{1}{2}, -\frac{1}{4} < m$\hnl
					-\frac{1}{12} - \frac{m}{2} + \frac{(1 + 4m)^{3/2}}{12} & $-\frac{1}{2} < x, -\frac{1}{4} < m$\nlwl{2pp1} \frac{x^{2p+1}}{2p+1} + \frac{x^{p+1}}{p+1}  & \left(\frac{-1 + \sqrt{1 + 4m}}{2}\right)^{1/p}\times\newline\frac{p}{2(2p+1)(p+1)}\times\newline
					\left(4(p+1)m+1-\sqrt{1+4m}\right) & $0 < x, 0 < m,\newline p\neq -\frac{1}{2}, -1, 0$\nl
					\frac{x^{2}}{2} + \frac{2x^{3/2}}{3}  & \frac{6m^{2} + 6m + 1 - (1+4m)^{3/2}}{12} & $0 < x, 0 < m$, this is Entry \ref{2pp1} with $p = \frac{1}{2}$\nl
					\frac{x^{5}}{5} + \frac{x^{3}}{3}  & \left(\frac{-1 + \sqrt{1 + 4m}}{2}\right)^{1/2}\times\newline\frac{1}{15}
					\left(12m+1-\sqrt{1+4m}\right) & $0 < x, 0 < m$, this is Entry \ref{2pp1} with $p = 2$\enl
				\end{tabular}$
			\end{center}
		\end{subtable}
	\end{table}
	\begin{table}[H]
		\begin{subtable}{\textwidth}
			\begin{center}$
				\begin{tabular}{|c|L{5cm}|L{5cm}|P{4cm}|}
					\firstrow
					ax - b & b & $m=a$\nl
					|x| & 0 & $|m|\leq 1$\nl
					-cx\text{ for } -a\leq x\leq 0 \newline\text{ and }
					x\text{ for }0 < x\leq b & -a(m+c)\text{ for }m\leq -c,\newline 0\text{ for } -c < m\leq 1,\text{ and}\newline b(m-1)\text{ for } 1 < m & \ \newline$0 < a, 0 < b, 0 < c$\nl
					R(x) = \max\{0, x\} & 0 & $0 < m\leq 1$, $R(x)$ is the \textit{unit ramp function}\nl
					\frac{x}{1-x} & (\sqrt{m}-1)^{2} & $x < 1, 0 < m$\hnl
					(\sqrt{m}+1)^{2} & $1 < x, 0 < m$\nl
					\frac{x^{2}}{x+1} = \frac{1}{x+1} + x - 1 & 2 - m + 2\sqrt{1-m} & $x < -1, m < 1$\hnl
					2 - m - 2\sqrt{1-m} & $-1 < x, m < 1$\nl
					\frac{x^{2}}{2}\text{ for } |x|<a, \text{ and }\newline
					a|x| - \frac{a^{2}}{2}\text{ for }|x|\geq a & \frac{m^{2}}{2} & $0\leq |m|\leq a, 0 < a, f(x)$ is the \textit{Huber loss function}\nl
					ax^{2} + bx + c&\frac{(m - b)^{2}}{4a} - c = \frac{m^{2} -2bm + b^{2}-4ac}{4a}&$a\neq 0$\nl
					ax^{3} + bx^{2} + cx + d & \frac{1}{27a^{2}}(2(3a(m - c) + b^{2})^{3/2} - b(9a(m-c) + 2b^{2})) - d & $-b/(3a) < x,\newline (3ac-b^{2})/(3a) < m,\newline 0 < a$
					\enl
				\end{tabular}$
			\end{center}
		\end{subtable}
	\end{table}
	
	\begin{table}[H]
		\begin{subtable}{\textwidth}
			\subtablelabel{Exponential and Logarithmic Functions}\label{exp}
			\begin{center}$
				\begin{tabular}{|c|L{5cm}|L{5cm}|P{4cm}|}
					\firstrowwl{ex}
					e^{x} & m\ln m - m & $0 < m$\nlwl{ax}
					a^{x} & \frac{m}{\ln a}\ln\left(\frac{m}{\ln a}\right) - \frac{m}{\ln a} & $0 < a, a\neq 1, \sign(m) = \sign(a-1)$\nl
					\frac{1}{1 + e^{-x}} = \frac{e^{x}}{1 + e^{x}} & m\ln\frac{1 - 2m - \sqrt{1 - 4m}}{2m} - \frac{1}{2}(1 - \sqrt{1 - 4m}) & $x < 0, 0 < m < \frac{1}{4}, (a + be^{-cx})^{-1}$ with $0 < a,b,c$ is the \textit{logistic} or \textit{growth curve} and the \textit{logistic cumulative distribution function}\hnl
					m\ln\frac{1 - 2m + \sqrt{1 - 4m}}{2m} - \frac{1}{2}(1 + \sqrt{1 - 4m}) & $0 < x, 0 < m < \frac{1}{4}$\nl
					\frac{1}{2}\frac{1 - e^{-x}}{1 + e^{-x}} = \frac{1}{1 + e^{-x}} - \frac{1}{2} & m\ln\frac{1 - 2m - \sqrt{1 - 4m}}{2m} + \frac{1}{2}\sqrt{1 - 4m} & $x < 0, 0 < m < \frac{1}{4}$\hnl
					m\ln\frac{1 - 2m + \sqrt{1 - 4m}}{2m} - \frac{1}{2}\sqrt{1 - 4m} & $0 < x, 0 < m < \frac{1}{4}$\nl
					\ln(x) & 1 + \ln m & $0 < x, 0 < m$\nlwl{xln}
					x\ln(x) & e^{m-1} & $0 < x$\nl
					\logit(x) = \ln \frac{x}{1 - x} & \frac{1}{2}(m - \sqrt{m(m-4)}) + \ln\left(\frac{1}{2}(m - 2 + \sqrt{m(m-4)})\right) & $0 < x < \frac{1}{2}, 4 < m$,\newline $\logit(x)$ is the \textit{log odds} or \textit{logit function}\hnl
					\frac{1}{2}(m + \sqrt{m(m-4)}) + \ln\left(\frac{1}{2}(m - 2 - \sqrt{m(m-4)})\right) & $\frac{1}{2} < x < 1, 4 < m$
					\enl
				\end{tabular}$
			\end{center}
		\end{subtable}
	\end{table}

	\begin{table}[H]
		\begin{subtable}{\textwidth}
			\begin{center}$
				\begin{tabular}{|c|L{5cm}|L{5cm}|P{4cm}|}
					\firstrow
					\ln \frac{x}{x + 1} & -\frac{1}{2}(\sqrt{m(m+4)} + m) + \ln\left(\frac{1}{2}(m + 2 - \sqrt{m(m+4)})\right) & $x < -1, 0 < m$\hnl
					\frac{1}{2}(\sqrt{m(m+4)} - m) + \ln\left(\frac{1}{2}(m + 2 + \sqrt{m(m+4)})\right) & $0 < x, 0 < m$\nl
					\ln(x) - \frac{1}{x} & \sqrt{1 + 4m} + \ln\left(\frac{1}{2}(\sqrt{1 + 4m} - 1)\right) & $0 < x, 0 < m$\nl
					\ln(x) + 2\sqrt{x} & \frac{2m - 1 - \sqrt{1 + 4m}}{2m} + 2\ln\left(\frac{1}{2}(\sqrt{1 + 4m} - 1)\right) & $0 < x, 0 < m$\nl
					\ln(x) + \frac{x^{2}}{2} & \frac{1}{4}(m^{2} + 2 - m\sqrt{m^{2} - 4}) + \ln\left(\frac{1}{2}(m + \sqrt{m^{2} - 4})\right) & $0 < x < 1, 2 < m$\hnl
					\frac{1}{4}(m^{2} + 2 + m\sqrt{m^{2} - 4}) + \ln\left(\frac{1}{2}(m - \sqrt{m^{2} - 4})\right) & $1 < x, 2 < m$\nl
					x(\ln(x))^{2} - 2x\ln x + 2x & -2e^{-\sqrt{m}}(\sqrt{m} + 1) & $0 < x < 1, 0 < m$\hnl
					2e^{\sqrt{m}}(\sqrt{m} - 1) & $1 < x, 0 < m$\nl
					\ln(1 + e^{x}) & m\ln m + (1 - m)\ln(1 - m) & $0 < m < 1, g(m)$ is the contribution to the \textit{Fermi-Dirac entropy} by a state with the probability $m$
					\enl
				\end{tabular}$
			\end{center}
		\end{subtable}
	\end{table}

	\begin{table}[H]
		\begin{subtable}{\textwidth}
			\begin{center}$
				\begin{tabular}{|c|L{5cm}|L{5cm}|P{4cm}|}
					\firstrow
					-\ln(1 - e^{x}) & m\ln m - (1 + m)\ln(1 + m) & $x < 0, 0 < m,$\newline $g(m)$ is the contribution of a state with $m$ particles to the \textit{Einstein-Bose entropy}\nl
					\ln(1 - e^{-x}) & (1 + m)\ln(1 + m) - m\ln m & $0 < x, 0 < m$\nl
					\sinh x & -m\ln(m + \sqrt{m^{2} - 1}) + \sqrt{m^{2} - 1} & $x < 0, 1 < m$\hnl
					m\ln(m + \sqrt{m^{2} - 1}) - \sqrt{m^{2} - 1} & $0 < x, 1 < m$\nl
					\cosh x & m\ln(m + \sqrt{m^{2} + 1}) - \sqrt{m^{2} + 1} & $x < 0, m < 0$ and $0 < x, 0 < m$\nl
					\tanh x & -m\ln\frac{1 + \sqrt{1 - m}}{\sqrt{m}} + \sqrt{1 - m} & $x < 0, 0 < m < 1$\hnl
					m\ln\frac{1 + \sqrt{1 - m}}{\sqrt{m}} - \sqrt{1 - m} & $0 < x, 0 < m < 1$\nl
					\sinh^{2} x & \frac{1}{2}(m\ln(m + \sqrt{m^{2} + 1}) - 1 - \sqrt{m^{2} + 1}) &\nl
					\cosh^{2} x & -(\sqrt{1 - m^{2}} + \ln m - \ln(1 + \sqrt{1 - m^{2}})) &\nl
					\ln \sinh x & \frac{m}{2}\ln\frac{1 + m}{1 - m} + \frac{1}{2}\ln(1 - m^{2}) & $0 < x, m < 1$\nl
					\ln \cosh x & \frac{m}{2}\ln\frac{1 + m}{1 - m} + \frac{1}{2}\ln(1 - m^{2}) & $x < 0, -1 < m < 0$ and $0 < x, 0 < m < 1$\nl
					\ln \tanh x & \frac{m}{2}\ln\frac{\sqrt{m^{2} + 4} + 2}{m} + \ln\frac{\sqrt{m^{2} + 4} + m}{2} & $0 < x, 0 < m$
					\enl
				\end{tabular}$
			\end{center}
		\end{subtable}
	\end{table}

	\begin{table}[H]
		\begin{subtable}{\textwidth}
			\begin{center}$
				\begin{tabular}{|c|L{5cm}|L{5cm}|P{4cm}|}
					\firstrow
					\sinh^{-1} x =\ln(x + \sqrt{x^{2} + 1}) & -\sqrt{1 - m^{2}} + \ln\frac{1 + \sqrt{1 - m^{2}}}{m} & $x < 0, 0 < m < 1$\hnl
					\sqrt{1 - m^{2}} + \ln\frac{1 - \sqrt{1 - m^{2}}}{m} & $0 < x, 0 < m < 1$\nl
					
					\cosh^{-1} x = \ln(x + \sqrt{x^{2} - 1}) & \sqrt{1 + m^{2}} - \ln\frac{1 + \sqrt{1 + m^{2}}}{m} & $1 < x, 0 < m$\nl
					
					\tanh^{-1} x = \frac{1}{2}\ln\frac{1+x}{1-x} & -\sqrt{m(m-1)} - \frac{1}{2}\ln(2m - 1 - 2\sqrt{m(m-1)})& $-1 < x < 0, 1 < m$\hnl
					\sqrt{m(m-1)} - \frac{1}{2}\ln(2m - 1 + 2\sqrt{m(m-1)})& $0 < x < 1, 1 < m$\nl
					
					\coth^{-1} x = \frac{1}{2}\ln\frac{x+1}{x-1} & \sqrt{m(m-1)} - \frac{1}{2}\ln(1 - 2m  - 2\sqrt{m(m-1)})& $x < -1, m < 0$\hnl
					-\sqrt{m(m-1)} - \frac{1}{2}\ln(1 - 2m + 2\sqrt{m(m-1)})& $1 < x, m < 0$
					\enl
				\end{tabular}$
			\end{center}
		\end{subtable}
	\end{table}

	\begin{table}[H]
		\begin{subtable}{\textwidth}
			\subtablelabel{Trigonometric Functions}\label{trig}
			\begin{center}$
				\begin{tabular}{|c|L{5cm}|L{5cm}|P{4cm}|}
					\firstrowwl{sin}
					\sin x & -m\Sin\sqrt{1 - m^{2}}+\sqrt{1-m^{2}} & $-\pi/2<x<0, 0<m<1$\hnl
					m\Sin\sqrt{1 - m^{2}}-\sqrt{1-m^{2}} & $0<x<\pi/2, 0<m<1$\nlwl{cos}
					
					\cos x & -m\Cos\sqrt{1 - m^{2}}-\sqrt{1-m^{2}} & $-\pi/2<x<0, 0<m<1$\hnl
					m\Cos\sqrt{1 - m^{2}}-\sqrt{1-m^{2}} & $0<x<\pi/2, -1<m<0$\hnl
					m\Cos\sqrt{1 - m^{2}}+\sqrt{1-m^{2}} & $\pi/2<x<\pi, -1<m<0$\nl
					
					\tan x & -m\Tan\sqrt{m - 1}+\sqrt{m-1} & $-\pi/2<x<0, 1<m$\hnl
					m\Tan\sqrt{m - 1}-\sqrt{m-1} & $0<x<\pi/2, 1<m$\nl
					
					\cot x = \frac{1}{\tan x} & -m\Sin\frac{1}{\sqrt{-m}}+\sqrt{-m-1} & $-\pi/2<x<0, m<-1$\hnl
					m\Sin\frac{1}{\sqrt{-m}}-\sqrt{-m-1} & $0<x<\pi/2, m<-1$\nl
					
					\sec x=\frac{1}{\cos x} & m\Cos\frac{\sqrt{\sqrt{4m^{2} + 1} - 1}}{\sqrt{2}m} - \frac{\sqrt{2}m}{\sqrt{\sqrt{4m^{2} + 1} - 1}} & $0<x<\pi/2, 0<m$\hnl
					m\Cos\frac{-\sqrt{\sqrt{4m^{2} + 1} - 1}}{\sqrt{2}m} + \frac{\sqrt{2}m}{\sqrt{\sqrt{4m^{2} + 1} - 1}} & $\pi/2<x<\pi, 0<m$
					\enl
				\end{tabular}$
			\end{center}
		\end{subtable}
	\end{table}

	\begin{table}[H]
		\begin{subtable}{\textwidth}
			\begin{center}$
				\begin{tabular}{|c|L{5cm}|L{5cm}|P{4cm}|}
					\firstrow
					\csc x=\frac{1}{\sin x} & m\Sin\frac{\sqrt{\sqrt{4m^{2} + 1} - 1}}{\sqrt{2}m} - \frac{\sqrt{2}m}{\sqrt{\sqrt{4m^{2} + 1} - 1}} & $-\pi/2<x<0, m<0$\hnl
					-m\Sin\frac{\sqrt{\sqrt{4m^{2} + 1} - 1}}{\sqrt{2}m} + \frac{\sqrt{2}m}{\sqrt{\sqrt{4m^{2} + 1} - 1}} & $0<x<\pi/2, m<0$\nl
					
					\sin^{2} x & \frac{1}{2}(m\Sin m + \sqrt{1 - m^{2}} - 1) & $-\pi/4 < x < \pi/4, -1<m<1$\nl
					\cos^{2} x & -\frac{1}{2}(m\Sin m + \sqrt{1 - m^{2}} + 1) & $-\pi/4 < x < \pi/4, -1<m<1$\nl
					2\sin\sqrt{x} - 2\sqrt{x}\cos\sqrt{x} & m(\Sin m)^{2} + 2\sqrt{1 - m^{2}}\Sin m - 2m & $0 < x < \pi^{2}/4, 0 < m < 1$\nl
					
					\sin x - \frac{1}{3}\sin^{3} x & -m\Sin((1 - m^{2/3})^{1/2}) + (1 - m^{2/3})^{1/2} - \frac{1}{3}(1 - m^{2/3})^{3/2} & $-\pi/2 < x < 0, 0 < m < 1$\hnl
					m\Sin((1 - m^{2/3})^{1/2}) - (1 - m^{2/3})^{1/2} + \frac{1}{3}(1 - m^{2/3})^{3/2} & $0 < x < \pi/2, 0 < m < 1$\nl
					
					-\cos x + \frac{1}{3}\cos^{3} x & m\Cos((1 - m^{2/3})^{1/2}) + (1 - m^{2/3})^{1/2} - \frac{1}{3}(1 - m^{2/3})^{3/2} & $0 < x < \pi/2, 0 < m < 1$\hnl
					m\Cos(-(1 - m^{2/3})^{1/2}) - (1 - m^{2/3})^{1/2} + \frac{1}{3}(1 - m^{2/3})^{3/2} & $\pi/2 < x < \pi, 0 < m < 1$\enl
				\end{tabular}$
			\end{center}
		\end{subtable}
	\end{table}
					
	\begin{table}[H]
		\begin{subtable}{\textwidth}
			\begin{center}$
				\begin{tabular}{|c|L{5cm}|L{5cm}|P{4cm}|}
					\firstrow
					\tan x + \frac{1}{3}\tan^{3} x & -m\Tan((m^{1/2} - 1)^{1/2}) + (m^{1/2} - 1)^{1/2} + \frac{1}{3}(m^{1/2} - 1)^{3/2} & $-\pi/2 < x < 0, 1 < m$\hnl
					m\Tan((m^{1/2} - 1)^{1/2}) - (m^{1/2} - 1)^{1/2} - \frac{1}{3}(m^{1/2} - 1)^{3/2} & $0 < x < \pi/2, 1 < m$\nl
					\ln\sin x = -\ln\csc x & m\Sin\frac{1}{\sqrt{1+m^{2}}} + \ln\sqrt{1 + m^{2}} & $0 < x < \pi/2, 0 < m$\nl
					
					\ln\cos x = -\ln\sec x & -m\Cos\frac{1}{\sqrt{1+m^{2}}} + \ln\sqrt{1 + m^{2}} & $-\pi/2 < x < 0, 0 < m$\hnl
					m\Cos\frac{1}{\sqrt{1+m^{2}}} + \ln\sqrt{1 + m^{2}} & $0 < x < \pi/2, m < 0$\nl
					
					\ln\tan x = -\ln\cot x & m\Tan\frac{m - \sqrt{m^{2} - 4}}{2} - \ln\frac{m - \sqrt{m^{2} - 4}}{2} & $0 < x < \pi/4, 2 < m$\hnl
					m\Tan\frac{m + \sqrt{m^{2} - 4}}{2} - \ln\frac{m + \sqrt{m^{2} - 4}}{2} & $\pi/4 < x < \pi/2, 2 < m$\nl
					
					\frac{x}{2}(\sin\ln x + \cos\ln x) & \frac{1}{2}e^{\Cos m}(m - \sqrt{1 - m^{2}}) & $1 < x < e^{\pi}, -1 < m < 1$\nl
					\frac{x}{2}(\sin\ln x - \cos\ln x) & \frac{1}{2}e^{\Sin m}(m + \sqrt{1 - m^{2}}) & $e^{-\pi/2} < x < e^{\pi/2}, -1 < m < 1$\nl
					
					\Sin x & -\sqrt{m^{2} - 1} + \Sin\frac{\sqrt{m^{2} - 1}}{m} & $-1 < x < 0, 1 < m$\hnl
					\sqrt{m^{2} - 1} - \Sin\frac{\sqrt{m^{2} - 1}}{m} & $0 < x < 1, 1 < m$
					\enl
				\end{tabular}$
			\end{center}
		\end{subtable}
	\end{table}

	\begin{table}[H]
		\begin{subtable}{\textwidth}
			\begin{center}$
				\begin{tabular}{|c|L{5cm}|L{5cm}|P{4cm}|}
					\firstrow
					\Cos x & \sqrt{m^{2} - 1} - \Cos\frac{\sqrt{m^{2} - 1}}{m} & $-1 < x < 0, m < -1$\hnl
					-\sqrt{m^{2} - 1} - \Cos\frac{-\sqrt{m^{2} - 1}}{m} & $0 < x < 1, m < -1$\nl
					
					\Tan x & \Tan\sqrt{\frac{1 - m}{m}} - \sqrt{m(1-m)} & $ x < 0, 0 < m < 1$\hnl
					-\Tan\sqrt{\frac{1 - m}{m}} + \sqrt{m(1-m)} & $0 < x, 0 < m < 1$\nl
					
					\Sin\sqrt{x} & \frac{m - \sqrt{m^{2} - 1}}{2} - \Sin\sqrt{\frac{m - \sqrt{m^{2} - 1}}{2m}} & $0 < x < 1/2, 1 < m$\hnl
					\frac{m + \sqrt{m^{2} - 1}}{2} - \Sin\sqrt{\frac{m + \sqrt{m^{2} - 1}}{2m}} & $1/2 < x < 1, 1 < m$\nl
					
					\Sin x - \sqrt{1 - x^{2}} & m - \Sin\frac{m^{2} - 1}{m^{2} + 1} & $-1 < x < 1, 0 < m$\nl
					
					\Sin\tanh x = \Tan\sinh x & -m\ln\frac{1 + \sqrt{1 - m^{2}}}{m} + \Sin\sqrt{1 - m^{2}} & $x < 0, 0 < m < 1$\hnl
					m\ln\frac{1 + \sqrt{1 - m^{2}}}{m} - \Sin\sqrt{1 - m^{2}} & $0 < x, 0 < m < 1$
					\enl
				\end{tabular}$
			\end{center}
		\end{subtable}
	\end{table}

	\begin{table}[H]
		\begin{subtable}{\textwidth}
			\subtablelabel{Special Functions}\label{spe}
			\begin{center}$
				\begin{tabular}{|c|L{5cm}|L{5cm}|P{4cm}|}
					\firstrow
					\erf(x) = \frac{2}{\sqrt{\pi}}\int_{0}^{x} e^{-t^{2}}dt & -m\sqrt{-\ln(\sqrt{\pi}m/2)} - \erf\left(-\sqrt{-\ln (\sqrt{\pi}m/2)}\right) & $x<0, 0<m<2/\sqrt{\pi}$, \newline $\erf(x)$ is the \textit{error function}\hnl
					m\sqrt{-\ln(\sqrt{\pi}m/2)} - \erf\left(\sqrt{-\ln (\sqrt{\pi}m/2)}\right) & $0<x, 0 < m < 2/\sqrt{\pi}$\nl
					
					\erfc(x) = 1 - \erf(x) = \frac{2}{\sqrt{\pi}}\int_{x}^{\infty} e^{-t^{2}}dt & -m\sqrt{-\ln(-\sqrt{\pi}m/2)} - \erfc\left(-\sqrt{-\ln (-\sqrt{\pi}m/2)}\right) & $x<0, -2/\sqrt{\pi}<m<0$, \newline $\erfc(x)$ is the \textit{complementary error function}\hnl
					m\sqrt{-\ln(-\sqrt{\pi}m/2)} - \erfc\left(\sqrt{-\ln (-\sqrt{\pi}m/2)}\right) & $0<x, -2/\sqrt{\pi}<m<0$\nl
					
					\erf^{-1}(x) & \sqrt{\ln(2m/\sqrt{\pi})} + m\,\erf\left(-\sqrt{\ln (2m/\sqrt{\pi})}\right) & $-1<x<0,\sqrt{\pi}/2<m$\hnl
					-\sqrt{\ln(2m/\sqrt{\pi})} + m\,\erf\left(\sqrt{\ln (2m/\sqrt{\pi})}\right) & $0<x<1,\sqrt{\pi}/2<m$\nl
					
					\erfc^{-1}(x) & -\sqrt{\ln(-2m/\sqrt{\pi})} + m\,\erfc\left(\sqrt{\ln (-2m/\sqrt{\pi})}\right) & $0<x<1,m<-\sqrt{\pi}/2$\hnl
					\sqrt{\ln(-2m/\sqrt{\pi})} + m\,\erfc\left(-\sqrt{\ln (-2m/\sqrt{\pi})}\right) & $1<x<2,m<-\sqrt{\pi}/2$
					\enl
				\end{tabular}$
			\end{center}
		\end{subtable}
	\end{table}
	
	\begin{table}[H]
		\begin{subtable}{\textwidth}
			\begin{center}$
				\begin{tabular}{|c|L{5cm}|L{5cm}|P{4cm}|}
					\firstrow
					\Phi(x) = \frac{1}{\sqrt{2\pi}}\int_{-\infty}^{x} e^{-t^{2}/2}dt & -m\sqrt{-\ln(2\pi m^{2})} - \Phi(-\sqrt{-\ln(2\pi m^{2})}) & $x < 0, 0 < m < 1/\sqrt{2\pi}$, \newline $\Phi(x)$ is the \textit{probability integral} or \textit{standard normal distribution function}\hnl
					m\sqrt{-\ln(2\pi m^{2})} - \Phi(\sqrt{-\ln(2\pi m^{2})}) & $0 < x, 0 < m < 1/\sqrt{2\pi}$\nl
					
					\frac{1}{\sqrt{2\pi}}\int_{x}^{\infty} e^{-t^{2}/2}dt & -m\sqrt{-\ln(2\pi m^{2})} - 1 + \Phi(-\sqrt{-\ln(2\pi m^{2})}) & $x < 0, -1/\sqrt{2\pi} < m < 0$, \newline $f(x)$ is the \textit{Q function}, which is $\Pr(X > x)$ for the standard normal distribution\hnl
					m\sqrt{-\ln(2\pi m^{2})} - 1 + \Phi(\sqrt{-\ln(2\pi m^{2})}) & $0 < x, -1/\sqrt{2\pi} < m < 0$\nl
					
					\Phi^{-1}(x) & \sqrt{\ln \frac{m^{2}}{2\pi}} + m\Phi\left(-\sqrt{\ln \frac{m^{2}}{2\pi}}\right) & $0 < x < 1/2, \sqrt{2\pi} < m$, \newline $\Phi^{-1}$ is the \textit{probit function} or the \textit{quantile function of the standard normal distribution}\hnl
					-\sqrt{\ln \frac{m^{2}}{2\pi}} + m\Phi\left(\sqrt{\ln \frac{m^{2}}{2\pi}}\right) & $1/2 < x < 1, \sqrt{2\pi} < m$\nl
				
					\frac{1}{\pi}\int_{-\infty}^{x} \frac{1}{1 + t^{2}}dt = \frac{1}{2} + \frac{1}{\pi}\Tan x & -m\sqrt{\frac{1 - \pi m}{\pi m}} + \frac{1}{\pi}\Tan\sqrt{\frac{1 - \pi m}{\pi m}} - \frac{1}{2} & $x<0, 0<m<1/\pi, f(x)$ is the \textit{standard Cauchy distribution} and \textit{Student's $t$-distribution} with one degree of freedom\hnl
					m\sqrt{\frac{1 - \pi m}{\pi m}} - \frac{1}{\pi}\Tan\sqrt{\frac{1 - \pi m}{\pi m}} - \frac{1}{2} & $0<x, 0<m<1/\pi$
					\enl				
				\end{tabular}$
			\end{center}
		\end{subtable}
	\end{table}

	\begin{table}[H]
		\begin{subtable}{\textwidth}
			\begin{center}$
				\begin{tabular}{|c|L{5cm}|L{5cm}|P{4cm}|}
					\firstrow
					
					\frac{\sqrt{2}}{4}\int_{-\infty}^{x}\left(1 + \frac{t^{2}}{2}\right)^{-3/2}dt
					= \frac{1}{2}+\frac{\sqrt{2}x}{4}\left(1 + x^{2}\right)^{-1/2} & \frac{1}{2}((1 - 2m^{2/3})^{3/2} - 1) & $x < 0, 0 < m < \sqrt{2}/4,$\newline $f(x)$ is \textit{Student's $t$-distribution} with two degrees of freedom\hnl
					-\frac{1}{2}((1 - 2m^{2/3})^{3/2} + 1) & $0 < x, 0 < m < \sqrt{2}/4$\nl
					
					\mathrm{S}(x) = \frac{1}{b}\int_{-\infty}^{x} \left(1 + \frac{t^{2}}{\nu}\right)^{-\frac{\nu + 1}{2}}dt & -m\sqrt{\nu\left(\frac{1}{(mb)^{\frac{2}{\nu+1}}} - 1\right)} - \mathrm{S}\left(-\sqrt{\nu\left(\frac{1}{(mb)^{\frac{2}{\nu+1}}} - 1\right)}\right) & $x < 0, 0 < m < \frac{1}{b},$\newline where $b = \sqrt{\nu}\mathrm{B}(1/2, \nu/2)$, $\mathrm{B}(x, y)$ is the beta function, $\mathrm{S}(x)$ is \textit{Student's $t$-distribution} with $\nu>0$ degrees of freedom\hnl
					m\sqrt{\nu\left(\frac{1}{(mb)^{\frac{2}{\nu+1}}} - 1\right)} - \mathrm{S}\left(\sqrt{\nu\left(\frac{1}{(mb)^{\frac{2}{\nu+1}}} - 1\right)}\right) & $0<x, 0 < m < \frac{1}{b},$\nl
					
					\mathrm{Ei}(x) = \int_{-\infty}^{x} \frac{e^{t}}{t}dt & -mW_{0}\left(-\frac{1}{m}\right) - \mathrm{Ei}\left(-W_{0}\left(-\frac{1}{m}\right)\right) & $x < 0, m < 0$ and $0 < x < 1, e < m$, $\mathrm{Ei}(x)$ is the exponential integral\hnl
					-mW_{-1}\left(-\frac{1}{m}\right) - \mathrm{Ei}\left(-W_{-1}\left(-\frac{1}{m}\right)\right) & $1 < x, e < m$\nl
					
					\mathrm{li}(x) = \int_{0}^{x}\frac{dt}{\ln t} & me^{1/m} - \mathrm{li}(e^{1/m}) \newline = me^{1/m} - \mathrm{Ei}\left(\frac{1}{m}\right) & $0 < x < 1, m < 0$ and $1 < x, 0 < m$, $\mathrm{li}(x)$ is the \textit{logarithmic integral}\nl
					
					\mathrm{Li}(x) = \mathrm{li}(x) - \mathrm{li}(2) = \int_{2}^{x}\frac{dt}{\ln t} & me^{1/m} - \mathrm{Li}(e^{1/m}) \newline = me^{1/m} - \mathrm{Ei}\left(\frac{1}{m}\right) + \mathrm{Ei}(\ln 2) 
					\newline = me^{1/m} - \mathrm{li}(e^{1/m}) + \mathrm{li}(2) & $1 < x, 0 < m$, $\mathrm{Li}(x)$ is the \textit{offset} or \textit{Eulerian logarithmic integral}
					
					\enl
				\end{tabular}$
			\end{center}
		\end{subtable}
	\end{table}

	\begin{table}[H]
	\begin{subtable}{\textwidth}
		\begin{center}$
			\begin{tabular}{|c|L{5cm}|L{5cm}|P{4cm}|}
				\firstrow
				F(x;k) = \int_{0}^{x}\frac{1}{\sqrt{1 - k^{2}\sin^{2}t}}dt & m\Sin\frac{\sqrt{m^{2} - 1}}{mk} - F\left(\Sin\frac{\sqrt{m^{2} - 1}}{mk}; k\right) & $0 < x < \pi/2, 1 < m < 1/\sqrt{1 - k^{2}}, 0 < k < 1$, $F(x;k)$ is the \textit{incomplete elliptic integral of the first kind}, $k$ is the \textit{elliptic modulus}\nl
				
				E(x;k) = \int_{0}^{x}\sqrt{1 - k^{2}\sin^{2}t}dt & m\Sin\frac{\sqrt{1 - m^{2}}}{k} - E\left(\Sin\frac{\sqrt{1 - m^{2}}}{k}; k\right) & $0 < x < \pi/2, \sqrt{1 - k^{2}} < m < 1, 0 < k < 1$, $E(x;k)$ is the \textit{incomplete elliptic integral of the second kind}, $k$ is the \textit{elliptic modulus}\nl
				
				J_{0}(x) & -mJ_{1}^{-1}(m) - J_{0}(J_{1}^{-1}(m)) & $|x| < 1.841\dots,$ $|m| < 0.582\dots,$ $J_{0}(x)$ and $J_{1}(x)$ are the \textit{Bessel functions of the first kind} of order $0$ and $1$\nl
				
				 \frac{x^{2}}{2}\left(\ln x - \frac{1}{2}\right) & \frac{m^{2}}{4}\frac{2W_{-1}(m) + 1}{W_{-1}^{2}(m)} & $0 < x < 1/e, -1/e < m < 0$\hnl
				 \frac{m^{2}}{4}\frac{2W_{0}(m) + 1}{W_{0}^{2}(m)} & $1/e < x, -1/e < m, m\neq 0, g(0) = 1/4$\nl
				 
				 x\ln x - x + \frac{x^{2}}{2} & \frac{W_{0}^{2}(e^{m})}{2} + W_{0}(e^{m}) & $0 < x$\nl
				 
				 e^{x} + \frac{x^{2}}{2} & \frac{1}{2}(m^{2} - W_{0}^{2}(e^{m}) - 2W_{0}(e^{m})) & \nl
				 
				 (x-1)e^{x} & m\left(W_{-1}(m) - 1 + \frac{1}{W_{-1}(m)}\right) = \int W_{-1}(m)dm & $x < -1, -1/e < m < 0$\hnl
				 m\left(W_{-0}(m) - 1 + \frac{1}{W_{0}(m)}\right) = \int W_{0}(m)dm & $-1 < x, -1/e < m, m\neq 0$
				 
				 \enl
				 
			\end{tabular}$
		\end{center}
	\end{subtable}
	\end{table}
	\begin{table}[H]
 	\begin{subtable}{\textwidth}
 		\begin{center}$
 			\begin{tabular}{|c|L{5cm}|L{5cm}|P{4cm}|}
 				\firstrow
 				
 				(x+a)e^{x} & m\left(W_{-1}(me^{a+1}) - a - 2\right) + e^{W_{-1}(me^{a+1}) - a - 1} & $x < -a-2, -1/e^{a+2} < m < 0$\hnl
 				m\left(W_{0}(me^{a+1}) - a - 2\right) + e^{W_{0}(me^{a+1}) - a - 1} & $-a-2 < x, -1/e^{a+2} < m$\nl
		 				
				\gamma(r, x) = \int_{0}^{x} t^{r-1}e^{-t}dt & (1-r)mW_{0}\left(\frac{m^{\frac{1}{r-1}}}{1-r}\right) - \gamma\left(r, (1-r)W_{0}\left(\frac{m^{\frac{1}{r-1}}}{1-r}\right)\right) & $0 < x < r-1,\newline 0 < m < (\frac{r-1}{e})^{r-1}, 1 < r$, $\gamma(r, x)$ is the \textit{lower incomplete gamma function}
				\hnl
				(1-r)mW_{-1}\left(\frac{m^{\frac{1}{r-1}}}{1-r}\right) - \gamma\left(r, (1-r)W_{-1}\left(\frac{m^{\frac{1}{r-1}}}{1-r}\right)\right) & $r - 1 < x,\newline 0 < m < (\frac{r-1}{e})^{r-1}, 1 < r$\nl
				 
				\Gamma(r, x) = \int_{x}^{\infty} t^{r-1}e^{-t}dt & -\L\{\gamma(r, x)\}(-m) - \Gamma(r) & $0 < x, 1 < r$, $\Gamma(r, x)$ is the \textit{upper incomplete gamma function}, $\Gamma(r)$ is the \textit{gamma function}
				\enl
			\end{tabular}$
		\end{center}
	\end{subtable}
\end{table}


\begin{thebibliography}{10}
		\expandafter\ifx\csname url\endcsname\relax
		\def\url#1{\texttt{#1}}\fi
		\expandafter\ifx\csname urlprefix\endcsname\relax\def\urlprefix{URL }\fi
		\expandafter\ifx\csname href\endcsname\relax
		\def\href#1#2{#2} \def\path#1{#1}\fi
		
		\bibitem{Arnold1978}
		V.~I. Arnold, {Mathematical Methods of Classical Mechanics}, 2nd Edition,
		Springer, New York, 1978.
		\newblock \href {https://doi.org/10.1007/978-1-4757-1693-1}
		{\path{doi:10.1007/978-1-4757-1693-1}}.
		
		\bibitem{doi:10.1119/1.3119512}
		R.~K.~P. Zia, E.~F. Redish, S.~R. McKay, {Making sense of the Legendre
			transform}, Am. J. Phys. 77~(7) (2009) 614--622.
		\newblock \href {https://doi.org/10.1119/1.3119512}
		{\path{doi:10.1119/1.3119512}}.
		
		\bibitem{Blackorby1978}
		C.~Blackorby, D.~Primont, R.~R. Russel, {Duality, Separability, and Functional
			Structure: Theory and Economic Applications}, Amsterdam, North-Holland, 1978.
		\newblock \href {https://doi.org/10.1016/0165-1889(81)90028-2}
		{\path{doi:10.1016/0165-1889(81)90028-2}}.
		
		\bibitem{RePEc:eee:jetheo:v:1:y:1969:i:4:p:374-396}
		L.~J. Lau,
		\href{https://ideas.repec.org/a/eee/jetheo/v1y1969i4p374-396.html}{{Duality
				and the structure of utility functions}}, J. Econ. Theory 1~(4) (1969)
		374--396.
		\newline\urlprefix\url{https://ideas.repec.org/a/eee/jetheo/v1y1969i4p374-396.html}
		
		\bibitem{hardy_1967}
		G.~H. Hardy, J.~E. Littlewood, G.~P{\'{o}}lya, {Inequalities}, Cambridge
		University Press, Cambridge, 1967.
		\newblock \href {https://doi.org/10.2307/3605504} {\path{doi:10.2307/3605504}}.
		
		\bibitem{Dragoslov1970}
		D.~S. Mitrinovi\'c, {Analytic Inequalities}, Springer-Verlag, Berlin, 1970.
		\newblock \href {https://doi.org/10.1007/978-3-642-99970-3}
		{\path{doi:10.1007/978-3-642-99970-3}}.
		
		\bibitem{Ince1944}
		E.~L. Ince, {Ordinary Differential Equations}, Dover Publications, New York,
		1944.
		
		\bibitem{Sternberg1954}
		S.~Sternberg, {Legendre transformations of curves}, Proc. Am. Math. Soc. 5~(6)
		(1954) 942--945.
		\newblock \href {https://doi.org/10.2307/2032562} {\path{doi:10.2307/2032562}}.
		
		\bibitem{Bix2006}
		R.~Bix, {Conics and Cubics: A Concrete Introduction to Algebraic Curves}, 2nd
		Edition, Springer, New York, 2006.
		\newblock \href {https://doi.org/10.1007/0-387-39273-4}
		{\path{doi:10.1007/0-387-39273-4}}.
		
		\bibitem{Glaeser2016}
		G.~Glaeser, H.~Stachel, B.~Odehnal, {The Universe of Conics: From the Ancient
			Greeks to 21st Century Developments}, Springer Spektrum, Berlin, 2016.
		\newblock \href {https://doi.org/10.1007/978-3-662-45450-3}
		{\path{doi:10.1007/978-3-662-45450-3}}.
		
		\bibitem{Rockafellar70convexanalysis}
		R.~T. Rockafellar, {Convex Analysis}, Princeton University Press, Princeton,
		1970.
		
		\bibitem{Lay2007}
		S.~R. Lay, {Convex Sets and Their Applications}, Dover Publications, New York,
		2007.
		
		\bibitem{Tiel1984ConvexAA}
		J.~van Tiel, {Convex Analysis: An Introductory Text}, John Wiley \& Sons,
		Chichester, 1984.
		
		\bibitem{Bauschke2017}
		H.~H. Bauschke, P.~L. Combettes, {Convex Analysis and Monotone Operator Theory
			in Hilbert Spaces}, 2nd Edition, Springer International Publishing, Cham,
		2017.
		
		\bibitem{bell2004special}
		W.~W. Bell, \href{https://books.google.com/books?id=Jj5pXGTZIKkC}{{Special
				Functions for Scientists and Engineers}}, Dover Publications, Mineola, 2004.
		\newline\urlprefix\url{https://books.google.com/books?id=Jj5pXGTZIKkC}
		
		\bibitem{lebedev1972special}
		N.~N. Lebedev, R.~A. Silverman,
		\href{https://books.google.com/books?id=po-6Yxz851MC}{{Special Functions and
				Their Applications}}, Dover Publications, Mineola, 1972.
		\newline\urlprefix\url{https://books.google.com/books?id=po-6Yxz851MC}
		
		\bibitem{50831}
		F.~W.~J. Olver, D.~W. Lozier, R.~F. Boisvert, C.~W. Clark, {NIST Handbook of
			Mathematical Functions}, Cambridge University Press, New York, 2010.
		
		\bibitem{Corless1996}
		R.~M. Corless, G.~H. Gonnet, D.~E.~G. Hare, D.~J. Jeffrey, D.~E. Knuth, {On the
			Lambert $W$ function}, Advances in Computational Mathematics 5~(1) (1996)
		329--359.
		\newblock \href {https://doi.org/10.1007/BF02124750}
		{\path{doi:10.1007/BF02124750}}.
		
		\bibitem{doi:10.1080/07468342.2004.11922095}
		E.~W. Packel, D.~S. Yuen, {Projectile motion with resistance and the Lambert
			$W$ function}, Coll. Math. J. 35~(5) (2004) 337--350.
		\newblock \href {https://doi.org/10.1080/07468342.2004.11922095}
		{\path{doi:10.1080/07468342.2004.11922095}}.
		
		\bibitem{10.1214/11-AOAS457}
		G.~M. Goerg, {Lambert $W$ random variables—a new family of generalized skewed
			distributions with applications to risk estimation}, Ann. Appl. Stat. 5~(3)
		(2011) 2197--2230.
		\newblock \href {https://doi.org/10.1214/11-AOAS457}
		{\path{doi:10.1214/11-AOAS457}}.
		
		\bibitem{10.1214/14-AOAS790}
		G.~M. Goerg, {Usage of the Lambert $W$ function in statistics}, Ann. Appl.
		Stat. 8~(4) (2014) 2567.
		\newblock \href {https://doi.org/10.1214/14-AOAS790}
		{\path{doi:10.1214/14-AOAS790}}.
		
		\bibitem{10.2112/JCOASTRES-D-17-00181.1}
		S.~Li, Y.~Liu, {Analytical and explicit solutions to implicit wave
			friction-factor equations based on the Lambert $W$ function}, J. Coast. Res.
		34~(6) (2018) 1499--1502.
		\newblock \href {https://doi.org/10.2112/JCOASTRES-D-17-00181.1}
		{\path{doi:10.2112/JCOASTRES-D-17-00181.1}}.
		
		\bibitem{doi:10.4169/074683410X480276}
		F.~Wang, {Application of the Lambert $W$ function to the SIR epidemic model},
		Coll. Math. J. 41~(2) (2010) 156--159.
		\newblock \href {https://doi.org/10.4169/074683410X480276}
		{\path{doi:10.4169/074683410X480276}}.
		
		\bibitem{Hayes2005}
		B.~Hayes, {Why $W$?}, Am. Sci. 93~(2) (2005) 104--108.
		
		\bibitem{Alford2019}
		M.~Alford,
		\href{https://web.physics.wustl.edu/alford/physics/Legendre\_introduction.pdf}{{Legendre
				transforms}} (2019).
		\newline\urlprefix\url{https://web.physics.wustl.edu/alford/physics/Legendre\_introduction.pdf}
		
		\bibitem{https://doi.org/10.1002/zamm.19630430916}
		E.~T. Whittaker, G.~N. Watson, {A Course of Modern Analysis}, 4th Edition,
		Cambridge University Press, London, 1996.
		\newblock \href {https://doi.org/10.1017/CBO9780511608759}
		{\path{doi:10.1017/CBO9780511608759}}.
		
		\bibitem{Abramowitz1972}
		M.~Abramowitz, I.~A. Stegun, {Handbook of Mathematical Functions with Formulas,
			Graphs, and Mathematical Tables}, Dover Publications, New York, 1972.
		\newblock \href {https://doi.org/10.1119/1.15378} {\path{doi:10.1119/1.15378}}.
		
		\bibitem{dienes1957taylor}
		P.~Dienes, {The Taylor Series: An Introduction to the Theory of Functions of a
			Complex Variable}, Dover Publications, Mineola, 1957.
		
		\bibitem{Grossman2005}
		N.~Grossman, {A $C^\infty$ Lagrange inversion theorem}, Am. Math. Mon. 112~(6)
		(2005) 512--514.
		\newblock \href {https://doi.org/10.1080/00029890.2005.11920222}
		{\path{doi:10.1080/00029890.2005.11920222}}.
		
		\bibitem{Johnson2002}
		W.~P. Johnson, {Combinatorics of higher derivatives of inverses}, Am. Math.
		Mon. 109~(3) (2002) 273--277.
		\newblock \href {https://doi.org/10.2307/2695356} {\path{doi:10.2307/2695356}}.
		
		\bibitem{Krattenthaler1988}
		C.~Krattenthaler, {Operator methods and Lagrange inversion: a unified approach
			to Lagrange formulas}, Trans. Am. Math. Soc. 305~(2) (1988) 431--465.
		\newblock \href {https://doi.org/10.2307/2000874} {\path{doi:10.2307/2000874}}.
		
		\bibitem{Traub1962}
		J.~F. Traub, {On the $n$th derivative of the inverse function}, Am. Math. Mon.
		69~(9) (1962) 904--907.
		\newblock \href {https://doi.org/10.2307/2311244} {\path{doi:10.2307/2311244}}.
		
		\bibitem{Rainville1964}
		E.~D. Rainville, {Elementary Differential Equations}, 3rd Edition, Macmillan,
		New York, 1964.
		
		\bibitem{zill2008differential}
		D.~G. Zill, M.~R. Cullen,
		\href{https://books.google.com/books?id=sXFUPgAACAAJ}{{Differential Equations
				with Boundary-Value Problems}}, 7th Edition, Brooks/Cole, Cengage Learning,
		Boston, 2008.
		\newline\urlprefix\url{https://books.google.com/books?id=sXFUPgAACAAJ}
		
		\bibitem{kilner2021tangent}
		S.~J. Kilner, D.~L. Farnsworth, \href{arxiv.org/abs/2105.11951}{{Tangent and
				supporting lines, envelopes, and dual curves}}, available at
		http://arxiv.org/abs/2105.11951, preprint. (2021).
		\newblock \href {http://arxiv.org/abs/2105.11951} {\path{arXiv:2105.11951}}.
		\newline\urlprefix\url{arxiv.org/abs/2105.11951}
		
		\bibitem{Thompson1996}
		A.~C. Thompson, {Minkowski Geometry}, Cambridge University Press, Cambridge,
		1996.
		\newblock \href {https://doi.org/10.1017/CBO9781107325845}
		{\path{doi:10.1017/CBO9781107325845}}.
		
	\end{thebibliography}
\end{document}